\documentclass[a4paper,12pt]{amsart}

\usepackage{latexsym,amsmath,amsfonts,amscd,amssymb,amsthm,mathrsfs}
\usepackage{tikz}
\usepackage{caption}
\usepackage{enumerate}
\usepackage[all]{xy}
\usepackage{color}
\usepackage{soul}
\usepackage{pdflscape}
\usepackage{booktabs}
\usepackage{longtable}
\usepackage[symbol]{footmisc}
\usepackage{multirow}
\usepackage{comment}
\usepackage{fullpage}
\usepackage{float}

\usepackage{hyperref}
\usepackage{ cleveref}

\def\u{\mathfrak{u}}

\def\g{\mathfrak{g}}
\def\h{\mathfrak{h}}

\def\R{\mathbb{R}}

\def\Z{\mathbb{Z}}
\def\N{\mathbb{N}}

\def\ad{\operatorname{ad}}
\def\tr{\operatorname{tr}}
\def\I{\operatorname{Id}}
\def\alt{\raise1pt\hbox{$\bigwedge$}}
\def\pint{\langle \cdotp,\cdotp \rangle }

\theoremstyle{plain}
\newtheorem{thm}{\bf Theorem}[section]
\newtheorem*{thm*}{\bf Theorem}
\newtheorem{cor}[thm]{\bf Corollary}
\newtheorem{prop}[thm]{\bf Proposition}

\theoremstyle{definition}
\newtheorem{defi}[thm]{\bf Definition}
\newtheorem{ejemplo}[thm]{\bf Example}

\theoremstyle{remark}
\newtheorem{obs}[thm]{\bf Remark}

\newcommand{\ri}{{\rm (i)}}
\newcommand{\rii}{{\rm (ii)}}

\title{On certain class of locally conformal symplectic structures of the second kind}

\author{M. Origlia}

\date{}

\address[Marcos Origlia]{KU Leuven Campus Kulak Kortrijk\\ E. Sabbelaan 53\\ BE-8500 Kortrijk, Belgium; and FaMAF-CIEM, Universidad Nacional de C\'{o}rdoba \\ X5000HUA C\'{o}rdoba\\ Argentina}
\thanks{This work was partially supported by the Research Foundation Flanders (Project G.0F93.17N) and CONICET, SECyTUNC and ANPCyT (Argentina)}
\email{marcosmiguel.origlia@kuleuven.be\\marcosoriglia@gmail.com}

\subjclass[2010]{53D05, 22E60, 22E25, 53C15, 22E40}
\keywords{Locally conformal symplectic structure, Lie algebra, Lie group, lattice, solvmanifold}

\begin{document}

\begin{abstract}

We study locally conformal symplectic (LCS) structures of the second kind on a Lie algebra.
We show a method to build new examples of Lie algebras admitting LCS structures of the second kind starting with a lower dimensional Lie algebra endowed with a LCS structure and a suitable representation. Moreover, we characterize all LCS Lie algebras obtained with our construction.
Finally, we study the existence of lattices in the associated simply connected Lie groups in order to obtain compact examples of manifolds admitting this kind of structure. 



\end{abstract}
\maketitle

\section{Introduction}

A \textit{locally conformal symplectic} 
structure (LCS for short) on the manifold $M$ is a non degenerate $2$-form $\omega$ such that 
there exists an open cover $\{U_i\}$ and smooth functions $f_i$ on $U_i$ such that 
$\omega_i=\exp(-f_i)\omega$ 
is a symplectic form on $U_i$. This condition is equivalent to requiring that
\begin{equation}\label{lcs}
d\omega=\theta\wedge\omega
\end{equation} 
for some closed $1$-form $\theta$, called the Lee form. 
The pair $(\omega, \theta)$ will be called a LCS structure on $M$.
It is well known that if $(\omega,\theta)$ is a LCS structure on $M$, then $\omega$ is symplectic if and only if $\theta=0$. 
Furthermore, $\theta$ is uniquely determined by equation \eqref{lcs}, but there is not an explicit formula for the Lee form.
If $\omega$ is a non degenerate $2$-form on $M$, with $\dim M\ge 6$, such that \eqref{lcs} holds for some $1$-form $\theta$ then $\theta$ is automatically closed and therefore $M$ is LCS.

\smallskip

According to Vaisman (see \cite{V}) there are two different types of LCS structures. If $(\omega, \theta)$ is a LCS 
structure on $M$, a vector field $X$ is called an infinitesimal automorphism of 
$(\omega,\theta)$ if $\textrm{L}_X\omega=0$, where $\textrm{L}$ denotes the Lie derivative. This 
implies $\textrm{L}_X\theta=0$ as well and, as a consequence, $\theta(X)$ is a constant function on $M$. We consider  $\mathfrak{X}_\omega(M)=\{X\in\mathfrak{X}(M): \textrm{L}_X\omega=0\}$ which is a subalgebra of $\mathfrak{X}_\omega(M)$, then the map $\theta|_{\mathfrak{X}_\omega(M)} : \mathfrak{X}_\omega(M) \to \R$ is a well defined Lie algebra morphism called the Lee morphism. 
If there exists an infinitesimal automorphism $X$ such that $\theta(X)\neq 0$, the LCS structure $(\omega,\theta)$ is said to be of {\em the first kind}, and it is of {\em the second kind} otherwise. This condition is equivalent to the Lee morphism being either surjective or identically zero.
In the literature, there is more information about LCS structures of the first kind, for example, in \cite{V} Vaisman gives relations with contact geometry and it is also proves that a manifold with a LCS structure of the first kind admits distinguished foliations. On the other hand, LCS structures of the second kind are less understood.

\smallskip

There is another way to distinguish LCS structures, to do this,
we can deform the de Rham differential $d$ to obtain the adapted 
differential operator 
\[d_\theta \alpha= d\alpha -\theta\wedge\alpha,\]
for any differential form $\alpha \in \Omega^*(M)$.
Since $\theta$ is $d$-closed, this operator satisfies $d_\theta^2=0$, thus it defines the {\em Morse-Novikov cohomology} 
$H_\theta^*(M)$ of $M$ relative to the closed $1$-form $\theta$ (see \cite{N1,N2}). Note that if $\theta$ is exact then $H_\theta^*(M)\simeq H_{dR}^*(M)$. It is known that if $M$ is a compact oriented $n$-dimensional manifold, then 
$H_\theta^0(M)= H_\theta^n(M)=0$ for any non exact closed $1$-form $\theta$ (see for instance 
\cite{GL,Ha}). For any LCS structure $(\omega,\theta)$ on $M$, the $2$-form $\omega$ 
defines a cohomology class $[\omega]_\theta\in H_\theta^2(M)$, since 
$d_\theta\omega=d\omega-\theta\wedge \omega=0$. 
The LCS structure $(\omega,\theta)$ is said to be {\em exact} if $\omega$ is $d_\theta$-exact or $[\omega]_\theta =0$, i.e., $\omega= d\eta-\theta\wedge \eta$ for some $1$-form $\eta$, and it is {\em non-exact} if $[\omega]_\theta \neq0$. 
It was proved in \cite{V} that if the LCS structure $(\omega,\theta)$
is of the first kind on $M$ then $\omega$ is $d_\theta$-exact, i.e., $[\omega]_\theta=0$. But the converse is not true. 
Recently in \cite{AOT} other cohomologies for LCS manifolds  were introduced, inspired by the almost Hermitian setting. More precisely, the authors define the LCS-Bott-Chern cohomology and the LCS-Aeppli cohomology on any compact LCS manifold, and compute them for some LCS solvmanifolds in low dimensions.

\

In last years, special attention has been devoted to the study of left invariant LCS structures on Lie groups (see for instance \cite{AO1,ABP,AOT,BM}), with very nice results in the case of LCS structures of the first kind. In this work we focus on LCS structures of the second kind on Lie algebras and solvmanifolds.



We recall that a LCS structure $(\omega,\theta)$ on a Lie group $G$ is called left invariant 
if $\omega$ is left invariant, and this easily implies that $\theta$ is also left invariant. 
Accordingly, we say that a Lie algebra $\g$ admits a {\em locally conformal symplectic} (LCS) 
structure if there exist $\omega\in\alt^2\g^*$ and $\theta \in \g^*$, with $\omega$ non degenerate 
and $\theta$ closed, such that \eqref{lcs} is satisfied. 

As in the case of manifolds we have that a LCS structure $(\omega,\theta)$ on a Lie algebra $\g$ 
can be of the first kind or of the second kind. Indeed, let us denote by $\g_\omega$ the set of 
infinitesimal automorphisms of the LCS structure, that is, 
\begin{equation}\label{autom}
\g_\omega = \{X\in\g: \textrm{L}_X\omega=0\} = \{X\in\g: \omega([X,Y],Z)+\omega(Y,[X,Z])=0 \, \text{for all} \, Y,Z\in\g\}
\end{equation}
where $\textrm{L}$ denotes the Lie derivative. Note that $\g_\omega \subset \g$ is a Lie subalgebra, thus the restriction of $\theta$ to 
$\g_\omega$ is a Lie algebra morphism called the {\em Lee morphism.} The LCS structure $ (\omega,\theta)$ is said to be {\em of the first kind} if the Lee morphism is surjective, and {\em of the second kind} if it is identically zero.

\smallskip

For a Lie algebra $\g$ and a closed $1$-form $\theta\in\g^*$ we also have the Morse-Novikov cohomology 
$H_\theta^*(\g)$ defined by the differential operator $d_\theta$ on $\alt^* \g^*$ defined by \[d_\theta \alpha= d\alpha 
-\theta\wedge\alpha.\] 


As in manifolds, we have that a LCS structure $(\omega,\theta)$ on a Lie algebra is said to be exact if $[\omega]_\theta =0$ or non-exact if $[\omega]_\theta\neq0$. 

It was proved in \cite{BM} that:
\begin{prop}\label{1kind=exact}
	If the Lie algebra $\g$ is unimodular, a LCS structure on $\g$ is of the first kind if and only if it is exact.
\end{prop}

\medskip

LCS structures of the first kind on Lie algebras are better understood because they are related with other important geometric structures (see for instance \cite{V} or more recently \cite{BM}). On the other hand not much is known about Lie algebras with a LCS structure of the second kind. 
In \cite{ABP} the authors study three different type of constructions of LCS Lie algebras. One corresponds to exacts LCS Lie algebras. The second one establishes a link between cosymplectic Lie algebras in dimension $2n-1$ and non-exact LCS structure in dimension $2n$. 
The third one is related to the existence of Lagrangian ideals (see \cite{ABP} for more details).
In \cite{AO1} we study LCS structures on almost abelian Lie groups, and we exhibit examples of solvmanifolds with LCS structure of the second kind in any dimension greater than or equal to $6$. We do not know many other explicit examples of solvmanifolds with a LCS structure of the second kind. Therefore, we consider that it would be very interesting to find new examples of Lie algebras and solvmanifolds admitting LCS structures of the second kind, hopefully they might be used to understand better this kind of structures.

In this work we deal with this problem and inspired by ideas of \cite{ABP} we provide another construction of LCS Lie algebras of the second kind which is different from those given in \cite{ABP}. 
After we provide this construction in Theorem \ref{pi=tita-ro} two question arise naturally:

\medskip

Question $1$: \textit{Given a LCS structure of the second kind on a Lie algebra, can it be obtained from our construction?} 

\medskip

Question $2$: \textit{Do there exist examples of LCS Lie algebras constructed by Theorem \ref{pi=tita-ro} such that the associated simply connected Lie group admits lattices?}

\medskip

We answer affirmatively question $2$. Indeed, we exhibit lattices for some associated simply connected Lie groups obtaining explicit examples of solvmanifolds admitting LCS structures of the second kind. 
And concerning question $1$, we give a nice characterization of Lie algebras built with our construction. Moreover, we recover most of the known examples in dimension $4$. 


 
\medskip

The outline of this article is as follows. In Section $2$ we prove that a left invariant LCS structure of the second kind on a Lie group induces a LCS structure of the second kind on any compact quotient by a discrete subgroup (see Theorem \ref{2tipo_cociente}). This allows us to study this geometric structure at the Lie algebra level.
In Section $3$  we give a method to build new examples of Lie algebras admitting LCS structures of the second kind starting with a Lie algebra endowed with a LCS structure and a compatible representation (see Theorem \ref{pi=tita-ro}), pointing out when the resulting Lie algebra is unimodular. 
We also have a converse (see Theorem \ref{converse}) obtaining a nice characterization of LCS Lie algebras admitting a non degenerate abelian ideal contained in the kernel of the Lee form.
In Section $4$ we show the wide range of our construction by reobtaining most of the known examples of LCS Lie algebras on the second kind in dimension $4$.
We also exhibit examples of Lie algebras in higher dimension starting with a $4$-dimensional LCS Lie algebra. Moreover we give a complete list of $4$-dimensional LCS Lie algebras which can be used to produce examples in higher dimension.
Finally, in Section $5$ we study the existence of lattices in the associated Lie groups and we give an explicit construction of a family of solvmanifolds admitting a LCS structure of the second kind proving that they are pairwise non homeomorphic.

\


\section{Solvmanifolds with LCS structure of the second kind}

Let us consider LCS structures on Lie algebras, or equivalently left invariant LCS structures on Lie groups. 
If the Lie group is simply connected then any left invariant LCS structure turns out to be globally conformal 
to a symplectic structure, which is essentially equivalent to having a symplectic structure on the Lie group. 
Therefore we will study compact quotients of such a Lie group by lattices.
Recall that a discrete subgroup 
$\Gamma$ of a simply connected Lie group $G$ is called a \textit{lattice} if the quotient 
$\Gamma\backslash G$ is compact. In this case  we have that 
$\pi_1(\Gamma\backslash G)\cong \Gamma$. The quotient $\Gamma\backslash G$ is called a solvmanifold if 
$G$ is solvable and it is called a nilmanifold if $G$ is nilpotent.

It is clear that a left invariant LCS structure on a Lie group $G$ induces a  LCS structure on any quotient $\Gamma\backslash G$, which will be non 
simply connected and therefore the inherited  LCS structure is ``strict''.

It is important to mention that if a Lie group admits a lattice then such Lie group must be unimodular, according to \cite{Mi}. Besides this necessary condition, there is no general criterion to determine whether a given  unimodular solvable Lie group admits a lattice. It is a very difficult problem in itself.
However, there is such a criterion for nilpotent Lie groups. Indeed, Malcev proved in \cite{Ma} that a nilpotent Lie group admits a lattice if and only if its Lie algebra has a rational form, that is, there exists a basis of the Lie algebra such that the corresponding structure constants are all rational. More recently, Bock studied in \cite{B} the existence of lattices in simply connected solvable Lie groups up to dimension 6 and he gave a criterion for the existence of lattices in almost abelian Lie groups.

Concerning LCS structures on a solvmanifold $\Gamma\backslash G$ which arise from a LCS structure on $\g=\text{Lie}(G)$, it is easy to see that if the LCS structure  on $\g$ is of the first kind, then the  induced LCS structure on the quotient $\Gamma\backslash G$ is of the first kind as well. 
We will prove in the next result that the same happens for LCS structures of the second kind, i.e., a LCS structure of the second kind on $\g$ induces a LCS structure of the second kind on any compact quotient $\Gamma\backslash G$.

\begin{thm}\label{2tipo_cociente}
	Let $\Gamma\backslash G$ be a solvmanifold and $\g=Lie(G)$. If $(\omega, \theta)$ is a LCS structure on $\g$ of the second kind, then the LCS structure induced on the solvmanifold $\Gamma\backslash G$ is of the second kind.
\end{thm}

\begin{proof}
	Let $(\omega,\theta)$ be a LCS structure of the second kind on $\g$. Since $\g$ is unimodular, it follows from \Cref{1kind=exact} that $(\omega,\theta)$ is not exact, i.e., $0\neq[\omega] \in H^2_\theta(\g)$. Let $(\hat\omega,\hat\theta)$ be the induced LCS structure on $\Gamma\backslash G$. According to \cite{K} there exists an injective map $i: H^2_\theta(\g) \to H^2_{\tilde\theta}(\Gamma\backslash G)$. This map arises from the natural inclusion of $\g$ into $\mathfrak X(\Gamma\backslash G)$.  Therefore  $0\neq[i(\omega)]=[\hat{\omega}] \in H^2_{\tilde\theta}(\Gamma\backslash G)$. Then $\hat\omega$ is non exact in $\Gamma\backslash G$ and it follows from \cite{V} that the LCS structure $(\hat\omega, \hat\theta)$ in the solvmanifold $\Gamma\backslash G$ is of the second kind.
\end{proof}

\begin{obs}
	Note that the LCS structure induced by Theorem \ref{2tipo_cociente} is indeed a non-exact LCS structure on $\Gamma\backslash G$.
\end{obs}

\begin{obs}
Note that in general a LCS structure of the second kind on a Lie algebra induces a LCS structure on the associated simply connected Lie group which is not necessarily of the second kind.
\end{obs}


\

\section{A method to construct LCS Lie algebras of the second kind}

In this section we give a method to build new examples of Lie algebras admitting a LCS structures of the second kind starting with a Lie algebra endowed with a LCS structure and a compatible representation (See Theorem \ref{pi=tita-ro}). Then we characterize all Lie algebras built with this method in Theorem \ref{converse}, we also point out when the Lie algebras constructed in Theorem \ref{pi=tita-ro} are unimodular which is a necessary condition to study lattices in the last section.

\

Let $\h$  be a Lie algebra, $(\omega,\theta)$ a LCS structure on $\h$, and let $(V,\omega_0)$ be a symplectic vector space of dimension $2n$. We consider a representation
\[\pi: \h \to \operatorname{End}(V).\]
Let $\g$ be the Lie algebra given by $\g=\h\ltimes_\pi V$, equipped with the non degenerate 2-form $\tilde\omega$ given by $\tilde\omega|_\h=\omega$ and $\tilde\omega|_V=\omega_0$. In particular $\tilde\omega(X,Y)=0$ for any $X\in\h$, $Y\in V$.
We define the $1$-form $\tilde\theta\in\g^*$ by $\tilde\theta|_\h=\theta$ and $\tilde\theta|_V=0$. 

We determine next when the pair $(\tilde\omega, \tilde\theta)$ is a LCS structure on the Lie algebra $\g$. Computing $d\tilde\omega=\tilde\theta\wedge\tilde\omega$ we can easily see that $(\tilde\omega, \tilde\theta)$ is a LCS structure if and only if the following condition is satisfied:
\begin{equation}\label{cons}
-\omega_0(\pi(X)Y,Z)+\omega_0(\pi(X)Z,Y)=\theta(X)\omega_0(Y,Z), 
\end{equation}
for $X\in\h$ and $Y,Z\in V$.

We denote by $S$ and $\rho$ the $\omega_0$-symmetric part and $\omega_0$-skew-symmetric part of $\pi$. More precisely, for each $X\in\h$, 
\[\pi(X)=S(X)+\rho(X),\]
where $S(X)$ is $\omega_0$-symmetric and $\rho(X)$ is $\omega_0$-skew-symmetric  with respect to the non degenerate $2$-form $\omega_0$, that is, $S(X)$ satisfies $\omega_0(S(X)Y,Z)=\omega_0(Y,S(X)Z)$ and $\rho(X)$ satisfies $\omega_0(\rho(X)Y,Z)=-\omega_0(Y,\rho(X)Z)$ for any $X\in\h$ and $Y,Z\in V$. This condition on $\rho$ is equivalent to saying that $\rho(X)\in\mathfrak{sp}(V,\omega_0)$ for any $X\in\h$.
It is easy to verify that \eqref{cons} holds if and only if $-2S(X)=\theta(X)\I$ for any $X\in\h$. 

\begin{defi}
With the notation above, if $\pi(X)=-\frac12\theta(X)\I+\rho(X)$ and $\rho(X)\in\mathfrak{sp}(V,\omega_0)$ for all $X\in\h$, then we say that $\pi$ is a \emph{LCS representation}.
\end{defi}
\begin{obs}
	Note that any LCS Lie algebra $\h$ admits a LCS representation. For example, taking $\rho=0$.
\end{obs}

Therefore we have the following result:

\begin{thm}\label{pi=tita-ro}
Let $\h$ be a Lie algebra with a LCS structure $(\omega,\theta)$, let $(V,\omega_0)$ be a $2n$-dimensional symplectic vector space and $\pi: \h \to \operatorname{End}(V)$ a representation. Let $\g=\h\ltimes_\pi V$ and $(\tilde\omega, \tilde\theta)$ given by $\tilde\omega|_\h=\omega$, $\tilde\omega|_V=\omega_0$, $\tilde\theta|_\h=\theta$ and $\tilde\theta|_V=0$. Then
$(\tilde\omega, \tilde\theta)$ is a LCS structure on $\g$ if and only if 
$\pi$ is a LCS representation. Moreover, $\rho: \h\to \mathfrak{sp}(V,\omega_0)$ is a representation and $(\tilde\omega, \tilde\theta)$ is a LCS structure of the second kind.
\end{thm}

\begin{proof}
As we mentioned above, $(\tilde\omega, \tilde\theta)$ is a LCS structure on $\g$ if and only if for any $X\in\h$ we have
$\pi(X)=-\frac12\theta(X)\I+\rho(X)$ with $\rho(X)\in\mathfrak{sp}(V,\omega_0)$. We have to check next that $\rho(X): \h\to \mathfrak{sp}(V,\omega_0)$ is a representation. We compute 
\begin{align*}
\pi([X,Y]) & = [\pi(X),\pi(Y)]\\
           & = [S(X)+\rho(X),S(Y)+\rho(Y)]\\
           & = [\rho(X),\rho(Y)],
\end{align*}
since $S(X)=-\frac12\theta(X)\I$. On the other hand we have that 
\begin{align*}
\pi([X,Y]) & = S([X,Y])+\rho([X,Y])\\
           & = -\frac12\theta([X,Y])+\rho([X,Y])\\
           & = \rho([X,Y]),
\end{align*}
since $\theta$ is closed and the first part of the result follows.

\smallskip
Finally, we see that $(\tilde\omega, \tilde\theta)$ is a LCS structure of the second kind. Indeed, let $X\in\g_{\tilde\omega}$, $X=H+Y$ with $H\in\h$ and $Y\in V$. Then $\tilde\omega([X,Z],W)+\tilde\omega(Z,[X,W])=0$ for any $Z,W\in\g$. In particular, for any $Z,W\in V$ we have:
	\begin{align*}
	0 & = \tilde\omega([X,Z],W)+\tilde\omega(Z,[X,W])\\
	& = \omega_0([H+Y,Z],W) + \omega_0(Z,[H+Y,W])\\
	& = \omega_0(\pi(H)Z,W) + \omega_0(Z,\pi(H)W)\\
	& = \omega_0(-\frac{1}{2}\theta(H)Z + \rho(H)Z,W) + \omega_0(Z, -\frac{1}{2}\theta(H)W + \rho(H)W)\\
	& = -\frac{1}{2}\theta(H)\omega_0(Z,W) + \omega_0(\rho(H)Z,W) -\frac{1}{2}\theta(H)\omega_0(Z,W) + \omega_0(Z,\rho(H)W)\\
	& = -\theta(H)\omega_0(Z,W),
	\end{align*}
	for any $Z,W\in V$, where we used that $\rho(H)$ is $\omega_0$-skew-symmetric in the last equality.
	Since $\omega_0$ is non degenerate on $V$ we can choose $Z,W\in V$ such that $\omega_0(Z,W)\neq0$, and therefore we get that $\theta(H)=0$ which implies $\theta(X)=0$. Then $\g_{\tilde\omega}\subset\ker\theta$, thus $\theta|_{\g_{\tilde\omega}}\equiv0$ and therefore we have that the LCS structure $(\tilde\omega,\tilde\theta)$ is of the second kind.
\end{proof}


\Cref{pi=tita-ro} provides us with a method to build new examples of Lie algebras equipped with a LCS structure of the second kind, starting with a LCS Lie algebra and a suitable representation. Note that the LCS structure on the initial Lie algebra can be of the first or of the second kind. We believe that this method is interesting because, as we mentioned before, there are not many general results about LCS structures of the second kind.

\begin{obs}
	Moreover, we can see that the LCS structure in Theorem \ref{pi=tita-ro} is non-exact. Indeed, suppose $(\tilde\omega, \tilde\theta)$ is an exact LCS structure. Then $\omega_0$, i.e. the restriction of $\tilde\omega$ to $\R^{2n}$, is zero, which is a contradiction since $\omega_0$ is non-degenerate.
\end{obs}

\begin{obs}
	This method can also be used to construct Lie algebras admitting other kind of structures. For example, if we start with a symplectic Lie algebra $(\h,\omega)$, that is $\theta=0$, then it is clear that we obtain a new symplectic Lie algebra $(\g,\tilde\omega)$.
\end{obs}

	The LCS Lie algebra $(\g,\tilde\omega, \tilde\theta)$ constructed in Theorem \ref{pi=tita-ro} has an abelian ideal which is non degenerate with respect to the restriction of the fundamental form $\tilde\omega$, namely, $(V,\omega_0)$. Moreover, it is contained in $\ker\tilde\theta$.
We show next a sort of converse of Theorem \ref{pi=tita-ro}. More precisely, we prove that any LCS Lie algebra with a non degenerate abelian ideal can be constructed as in Theorem \ref{pi=tita-ro}.

\

Let $\g$ be a Lie algebra endowed with a LCS structure $(\omega', \theta')$. Let $\u$ be a non degenerate ideal, and we consider the complement $\u^\perp$ given by $\u^\perp=\{X\in \g:  \omega'(X,U)=0, \forall U\in\u\}$. Since $\u$ is non degenerate we have that $\g=\u^\perp\oplus\u$ as vector spaces.

\begin{thm}\label{converse}
	Let $(\g,\omega',\theta')$ be a LCS Lie algebra admitting a non degenerate ideal $\u$: 
	\begin{enumerate}[(i)]
		\item if $\u\subset\ker\theta'$, then $\u^\perp$ is a subalgebra;
		\item if, moreover, $\u$ is abelian, then $\ad: \u^\perp \to \operatorname{End}(\u)$ is a LCS representation and $(\g,\omega',\theta')$ is isomorphic to $(\u^\perp\ltimes\u,\tilde\omega,\tilde\theta)$ with the LCS structure of Theorem \ref{pi=tita-ro}. In particular, $(\omega',\theta')$ is a LCS structure is of the second kind.
	\end{enumerate}	
\end{thm}
\begin{proof}
 $\ri$ We show $\u^\perp$ is a subalgebra, that is $\omega'([X,Y],U)=0$ for all $U\in \u$ and $X,Y\in \u^\perp$. We compute
 \begin{align*}
 \omega'([X,Y],U) & = -d\omega'(X,Y,U) -\omega'([Y,U],X)-\omega'(U,[X,Y])\\ 
 & = -\theta'(X)\omega'(Y,U)-\theta'(Y)\omega'(U,X)-\theta'(U)\omega'(X,Y)\\
 &=0,
 \end{align*}
 where we used the LCS condition and the fact that $\u\subset\ker\theta'$, thus $\ri$ is proved. 
 
 \medskip
 
 $\rii$ Let $(\omega,\theta)$ be the restriction of $(\omega',\theta')$ to $\u^\perp$. It is clear that $\omega$ is non degenerate on $\u^\perp$ and $\theta\neq0$ since $\u\subset\ker\theta'$ and $\theta'\neq0$. Therefore, $(\omega,\theta)$ satisfies the LCS condition on the subalgebra $\u^\perp$. 
 
We can decompose $\g$ as a semidirect product $\g=\u^\perp\ltimes\u$, and we denote the restriction of $\omega'$ to the non degenerate abelian ideal $\u$ by $\omega_0$. Then, it is clear that $\omega_0$ and $\ad: \u^\perp \to \operatorname{End}(\u)$ satisfy condition \eqref{cons}, since $(\omega',\theta')$ is exactly the LCS structure $(\tilde\omega,\tilde\theta)$  constructed in Theorem \ref{pi=tita-ro} with initial data $(\u^\perp, \omega, \theta)$ and $(\u, \omega_0)$. 
 
 Therefore, it follows from Theorem \ref{pi=tita-ro} that $\ad: \u^\perp \to \operatorname{End}(\u)$ is a LCS representation and $(\g,\omega',\theta')$ is isomorphic to $(\u^\perp\ltimes\u,\tilde\omega,\tilde\theta)$. Moreover, $(\omega',\theta')$ is a LCS structure of the seconk kind on $\g$.
\end{proof}


      
\begin{obs}
Using Theorem \ref{converse} it is easy now to determine whether a given Lie algebra endowed with a LCS structure can be constructed by Theorem \ref{pi=tita-ro}. This gives us a nice characterization of LCS Lie algebras admitting a non degenerate abelian ideal contained in the kernel of the Lee form. 	One can see Theorem \ref{converse} as a kind of reduction of the LCS condition. 
\end{obs} 

\begin{obs}
 This construction has some similar ideas to the ones in \cite[Proposition 1.17]{ABP} since both are related with special abelian ideals contained in the kernel of the Lee form. In our work we look for a non degenerate abelian ideal instead of a Lagrangian abelian ideal.
\end{obs} 



 
\

\subsection{Center of LCS Lie algebras of the second kind} 

\


The center of a Lie algebra with a LCS structure was studied in \cite{ABP}, in particular the authors characterized the center of a nilpotent LCS Lie algebra and they proved that the dimension of the center is at most $2$. Note that the nilpotency condition implies that the LCS structure is of the first kind. On the other hand, as a consequence of \Cref{pi=tita-ro} it is easy to verify that there is no restriction for the dimension of the center of a Lie algebra admitting a LCS structure of the second kind, as we show in the following example.

\begin{ejemplo}
Consider the $4$-dimensional Lie algebra $\mathfrak r\mathfrak r_{3,\lambda}$  with structure constants $(0,-12,-\lambda 13,0)$.
It means that we fix a coframe $\{e^1,e^2,e^3,e^4\}$ for $(\mathfrak r\mathfrak r_{3,\lambda})^*$ such that $de^1=0$, $de^2=-e^1\wedge e^2$, $de^3=-\lambda e^1\wedge e^3$ and $de^4=0$.
According to \cite{ABP} this Lie algebra admits a LCS structure given by $\omega=e^{12}+ e^{34}$ with Lee form $\theta=-\lambda e^1$, where $\{e^1,e^2,e^3,e^4\}$ denotes the dual coframe. Consider now the $(2n+4)$-dimensional Lie algebra $\g=\mathfrak r\mathfrak r_{3,\lambda}\ltimes_\pi\R^{2n}$ where $\pi(e_i)=0$ for $i=2,3,4$ and $\pi(e_1)\in M(2n,\R)$ is given by 
\[\pi(e_1)=\begin{pmatrix}
\frac\lambda 2\I_{n\times n} & \\
&\frac\lambda 2\I_{n\times n}\\
 \end{pmatrix}+
\begin{pmatrix}
\frac\lambda 2\I_{n\times n} & \\
&-\frac\lambda 2\I_{n\times n}\\
\end{pmatrix}
=\begin{pmatrix}
\lambda \I_{n\times n} & \\
&\operatorname{0}_{n\times n}\\
\end{pmatrix},\]
in a basis $\{u_1,\dots,u_n,v_1,\dots,v_n\}$ of $\R^{2n}$. More precisely, the Lie brackets on $\g$ are:
$$[e_1,e_2]=e_2, \quad [e_1,e_3]=\lambda e_3, \quad [e_1,u_k]=\lambda u_k,$$
for $k=1,\dots,n$. Then $e_4\in\mathfrak z(\g)$ and $v_i \in\mathfrak z(\g)$ for $i=1,\dots,n$, therefore the dimension of $\mathfrak z(\g)$ is $n+1$.
Moreover, it can be seen that $\pi$ is a LCS representation, hence according to \Cref{pi=tita-ro} it determines a LCS structure on the Lie algebra $\g$ given by $\tilde\omega=e^{12}+ e^{34} + \sum_{i=1}u^i\wedge v^i$ with Lee form $\tilde\theta=-\lambda e^1$. 
\end{ejemplo}

\

\subsection{Unimodular LCS Lie algebras of the second kind}

\

Since we are interested in finding examples of solvmanifolds equipped with a LCS structure of the second kind, we determine next when a Lie algebra $\g$ built in Theorem \ref{pi=tita-ro} is unimodular.

\begin{prop}\label{unimodularLCS}
Let $\h$ be a Lie algebra with a LCS structure  $(\omega,\theta)$, $(\pi,V)$ a $2n$-dimensional LCS representation of $\h$ and $\g=\h\ltimes_\pi V$ the Lie algebra with LCS structure  $(\tilde\omega,\tilde\theta)$ built as above. Then $\g$ is unimodular if and only if  $\tr(\ad_X^\h)=n\theta(X)$ for any $X\in\h$.
\end{prop}

\begin{proof}
Given $X\in\h$, the operator $\ad_X^\g:\g\to\g$ can be written as
\[\ad_X^\g=
\left(\begin{array}{c|c}
\ad_X^\h & \\
\hline
& \pi(X) \\
\end{array}\right),\]
for some bases of $\h$ and $V$. Then we have that $\g$ is unimodular if and only if $\tr(\pi(X))=-\tr(\ad_X^\h)$ for all $X\in\h$, and using the characterization in \Cref{pi=tita-ro} this happens if and only if \[\tr(\ad_X^\h)= -\tr\left(-\frac12\theta(X)\I+\rho(X)\right)=n\theta(X),\] 
for any $X\in\h$. 
\end{proof}



In particular we have the following corollary, which will be used later:

\begin{cor}
Let $\h$ be a Lie algebra with a LCS structure  $(\omega,\theta)$. If there exists $n \in \N$ such that
\begin{equation}\label{unimodular-condition}
\tr(\ad_X^\h)=n\theta(X)
\end{equation}
for all $X\in\h$,
then for any LCS representation $(\pi, V)$ with $\dim V=2n$, the LCS Lie algebra $\g=\h\ltimes_\pi V$ is unimodular.
\end{cor} 

Therefore the method of Theorem \ref{pi=tita-ro} together with condition  \eqref{unimodular-condition} allow us to build unimodular Lie algebras admitting a LCS structure of the second kind starting from a non unimodular Lie algebra with a LCS structure.

\

\section{Examples of Lie algebras with LCS structures of the second kind}

In this section we show first that our construction is quite general. Indeed, we can reobtain with this construction most of the known examples of Lie algebras admitting a LCS structure of the second kind. 
More precisely, we see that every unimodular $4$-dimensional Lie algebra admitting a LCS structure of the second kind has a non degenerate abelian ideal contained in the kernel of the Lee form, and then by Theorem \ref{converse} they can be obtained with our construction. 
We also use Theorem \ref{pi=tita-ro} to construct new examples of unimodular Lie algebras admitting LCS structures of the second kind in higher dimension.

\


\subsection{Dimension $4$}

We start by recalling in Table \ref{4uniLCS} the unimodular $4$-dimensional Lie algebras admitting a LCS structure of the second kind (see \cite{ABP}).

\begin{table}[h]
	\def\arraystretch{1}
	\centering
	{\resizebox{\textwidth}{!}{
			\begin{tabular}{l|l|c}
				Lie algebra &  structure equations & LCS structure of the second kind\\
				\toprule 
				\hline
				
				$\mathfrak r\mathfrak r_{3,-1}$ &  	\multirow{2}{*}{$(0,-12, 13,0)$} & $\theta=  e^1$  \\
				&&$\omega=e^{12}+e^{34}$ \\
				\hline

				&  \multirow{6}{*}{$(14,\alpha 24,-(1+\alpha) 34,0)$} & $\theta= \alpha e^4$,  $\alpha\neq -\frac12$ \\
				$\mathfrak r_{4,\alpha,-(1+\alpha)}$&&$\omega=e^{13}+e^{24}$\\
				\cline{3-3}
				{\small $-1<\alpha\leq-(1+\alpha)\leq1$}&  & $\theta= e^4$  \\
				{\small $\alpha\neq 0$}& & $\omega=e^{14}+e^{23}$\\
				\cline{3-3}
				& & $\theta= -(1+\alpha)e^4$  \\
				&&$\omega=e^{12}+e^{34}$ \\
				\hline

				$\mathfrak r'_{4,-\frac12,\delta}$ &  	\multirow{2}{*}{$(14,-\frac12 24+\delta 34,-\delta 24-\frac12 34,0)$} & $\theta=  e^4$  \\
				{\small$\delta>0$} & & $\omega=e^{14}\pm e^{23}$ \\
				\hline
				
				&  \multirow{4}{*}{$(14,-24,-12,0)$} & $\theta= e^4$ \\
				$\mathfrak d_4$&&$\omega=e^{12}-e^{34}+e^{24}$ \\
				\cline{3-3}
				& & $\theta= e^4$  \\
				& & $\omega=\pm e^{14}+e^{23}$\\
				
				\hline

				\bottomrule
			\end{tabular}
	}}
	\caption{Unimodular $4$-dimensional Lie algebras admitting LCS structure of second kind}
	\label{4uniLCS}
\end{table}

We explain in details the first case in Table \ref{4uniLCS}. Let $\g=\mathfrak r\mathfrak r_{3,-1}$ with the LCS structure 
$\omega=e^{12}+e^{34}, \theta= e^1$. It is clear that $\u=\text{span}\{e_3,e_4\}$ is a non degenerate abelian ideal of $\g$. It follows from Theorem \ref{converse} that $\g=\u^\perp\ltimes\u$ where $\u^\perp=\text{span}\{e_1,e_2\}$ is isomorphic to the $2$-dimensional non abelian Lie algebra $\mathfrak{aff}(\R)$. The LCS structure on $\u^\perp$ is given by $\omega=e^{12}, \theta=  e^1$, note that since $\dim\u^\perp=2$, this structure is in fact a symplectic structure. Therefore $\mathfrak r\mathfrak r_{3,-1}\simeq\mathfrak{aff}(\R)\ltimes_{\pi_1} \R^{2}$ with $\pi_1(e_1)=\operatorname{diag}(-1,0)=\operatorname{diag}(-\frac12,-\frac12)+\operatorname{diag}(-\frac12,\frac12)$ and $\pi_1(e_2)=0$. Clearly, $\pi_1$ is an LCS representation.

\smallskip

It can be easily seen that for any of the Lie algebras and any of the LCS structures of Table \ref{4uniLCS} we can proceed in the same way with the exception of $\mathfrak d_4$ with the LCS structure $\omega=e^{12}-e^{34}+e^{24}$ and $\theta= e^4$. Indeed, this is the only LCS structure of the second kind on a $4$-dimensional unimodular Lie algebra which does not satisfy the condition of Theorem \ref{converse}, that is, it does not have a non degenerate abelian ideal contained in $\ker\theta$. 

To summarize we have the following result:

\begin{prop}
	Any unimodular $4$-dimensional LCS Lie algebra of the second kind, with the only exception of $(\mathfrak d_4, \omega=e^{12}-e^{34}+e^{24}, \theta=e^4)$, can be reobtained by Theorem \ref{pi=tita-ro} for a suitable representation.
\end{prop}

According to \cite{ABP} the simply connected Lie groups associated with the Lie algebras of Table \ref{4uniLCS} (for a countable set of parameters $\alpha$ and $\delta$) admit lattices (see also \cite{B}).
Then we have that:

\begin{cor}
	Any unimodular  $4$-dimensional LCS Lie algebra of the second kind (except for $(\mathfrak d_4, e^{12}-e^{34}+e^{24}, e^4)$) associated with a compact solvmanifold can be obtained by Theorem \ref{pi=tita-ro} for a suitable representation.
\end{cor}

Moreover, it follows from Theorem \ref{2tipo_cociente} that the induced LCS structures on any quotient are LCS structures of the second kind. 

\

\subsection{Higher dimension}

We focus now on provide examples in higher dimension.
In order to build examples of unimodular Lie algebras admitting LCS structures we need to start with a non unimodular Lie algebra with a LCS structure in lower dimension.
In \cite{ABP} the authors classify the $4$-dimensional solvable Lie algebras admitting LCS structures up to automorphism of the Lie algebra. Using their classification we show in Table \ref{extensibles} all the $4$-dimensional solvable Lie algebras with their associated LCS structure satisfying condition \eqref{unimodular-condition}, which means that they can be extended to a higher dimensional unimodular Lie algebra admitting a LCS structure given by \Cref{pi=tita-ro}.

\begin{table}[h]
	\def\arraystretch{1.4}
	\centering
	{\resizebox{\textwidth}{!}{
			\begin{tabular}{l|l|l|l|c}
				Lie algebra &  structure equations & LCS structure & parameters & $2n$-extension  \\
				\toprule 
				\hline
				
				$\mathfrak r\mathfrak r_{3,\lambda}$ &  	\multirow{2}{*}{$(0,-12,-\lambda 13,0)$} & $\theta= -\lambda e^1$  & \multirow{2}{*}{$\lambda\neq 0$} &  \multirow{2}{*}{$n=-\frac{1+\lambda}{\lambda}$} \\
				{\small$\lambda\neq -1$}&&$\omega=e^{12}+e^{34}$ &&\\
				\hline
				
				\multirow{2}{*}{$\mathfrak r_2\mathfrak r_2$} &  	\multirow{2}{*}{$(0,-12,0,-34)$} & $\theta= -\sigma(e^1+e^3)$  & \multirow{2}{*}{$\sigma\neq -\frac12, 0$} &  \multirow{2}{*}{$n=\frac{1}{\sigma}$} \\
				&&$\omega=-\frac{\sigma+1}{\sigma}e^{12}+e^{14}+e^{23}+\frac{\sigma+1}{\sigma}e^{34}$ &&\\
				\hline
				
				\multirow{2}{*}{$\mathfrak r'_2$} &  	\multirow{2}{*}{$(0,0,-13+24,-14-23)$} & $\theta= \sigma e^1$  & \multirow{2}{*}{$\sigma\neq -1,0$} &  \multirow{2}{*}{$n=\frac{2}{\sigma}$} \\
				&&$\omega=e^{13}-\frac{1}{\sigma+1}e^{24}$ &&\\
				\hline
				
				\multirow{4}{*}{$\mathfrak r_{4,\mu}$} &  	\multirow{4}{*}{$(14,\mu24+34,\mu34,0)$} & $\theta= -(\mu+1)e^4$  & \multirow{2}{*}{$\mu\neq -1,1$} &  \multirow{2}{*}{$n=-\frac{2\mu+1}{\mu+1}$} \\
				&&$\omega=e^{13}+e^{24}$ &&\\\cline{3-5}
				&& $\theta= -2\mu e^4$  & \multirow{2}{*}{$\mu\neq 0,1$} &  \multirow{2}{*}{$n=-\frac{2\mu+1}{2\mu}$} \\
				&&$\omega=e^{14}\pm e^{23}$ &&\\
				\hline
				
				&  \multirow{6}{*}{$(14,\alpha 24,\beta 34,0)$} & $\theta= -(1+\beta)e^4$  & \multirow{2}{*}{$\alpha\neq \beta$} &  \multirow{2}{*}{$n=-\frac{1+\alpha+\beta}{1+\beta}$} \\
				$\mathfrak r_{4,\alpha,\beta}$&&$\omega=e^{13}+e^{24}$ &&\\
				\cline{3-5}
				{\small $-1<\alpha\leq\beta\leq1$}&  & $\theta= -(\alpha+\beta)e^4$  & \multirow{2}{*}{$ \beta\neq 1$} &  \multirow{2}{*}{$n=-\frac{1+\alpha+\beta}{\alpha+\beta}$} \\
				{\small $\alpha\beta\neq 0$}& & $\omega=e^{14}+e^{23}$ & $\alpha+\beta\neq 0$&\\
				\cline{3-5}
			    & & $\theta= -(1+\alpha)e^4$  & \multirow{2}{*}{$\forall\alpha, \beta$} & \multirow{2}{*}{$n=-\frac{1+\alpha+\beta}{1+\alpha}$}\\
				&&$\omega=e^{12}+e^{34}$ &&\\
				\hline
				
				$\hat{\mathfrak r}_{4,\beta}$ &  	\multirow{2}{*}{$(14,-24,\beta34,0)$} & $\theta= -(1+\beta )e^4$  & \multirow{2}{*}{$\beta\neq -1$} &  \multirow{2}{*}{$n=-\frac{\beta}{1+\beta}$} \\
				{\small$-1\leq\beta<0$} & & $\omega=e^{13}+e^{24}$ &&\\
				\hline
				
		     	$\mathfrak r'_{4,\gamma,\delta}$ &  	\multirow{2}{*}{$(14,\gamma 24+\delta 34,-\delta 24+\gamma 34,0)$} & $\theta= -\gamma e^4$  & \multirow{2}{*}{$\gamma\neq 0$} &  \multirow{2}{*}{$n=-\frac{1+2\gamma}{2\gamma}$} \\
				{\small$\delta>0$} & & $\omega=e^{14}\pm e^{23}$ &&\\
				\hline

				&  \multirow{6}{*}{$(\lambda14,(1-\lambda)24,-12+34,0)$} & $\theta= \sigma e^4$  & $\lambda\neq 2$ &  \multirow{2}{*}{$n=\frac{2}{\sigma}$} \\
				& & $\omega=e^{12}-(\sigma +1)e^{34}$ & $\sigma\neq -1,0$ &\\
				\cline{3-5}
				$\mathfrak d_{4,\lambda}$ &  & $\theta= (\lambda-1)e^4$  & \multirow{2}{*}{$ \lambda\neq \frac12, 1$} &  \multirow{2}{*}{$n=\frac{2}{\lambda-1}$} \\
				{\small $\lambda\geq\frac12$} & & $\omega=e^{12}-\lambda e^{34}+e^{24}$ &&\\
				\cline{3-5}
				& & $\theta= (\lambda-2)e^4$  & \multirow{2}{*}{$ \lambda\neq \frac12, 2$} & \multirow{2}{*}{$n=\frac{2}{\lambda-2}$}\\
				&&$\omega=e^{14}\pm e^{23}$ &&\\
				\hline
								
				\multirow{2}{*}{$\mathfrak d'_{4,\delta}$} &  	\multirow{2}{*}{$(\frac{\delta}{2}14+24,-14+\frac{\delta}{2}24,-12+\delta34,0)$} & $\theta= \sigma e^4$  & $\delta>0$ &  \multirow{2}{*}{$n=\frac{2\delta}{\sigma}$} \\
				&&$\omega=\pm(e^{12}-(\delta+\sigma)e^{34})$ & $\sigma\neq -\delta, 0$ &\\
				\hline					
	          
	   			\multirow{2}{*}{$\mathfrak h_4$} &  	\multirow{2}{*}{$(\frac{\delta}{2}14+24,\frac12 24,-12+34,0)$} & $\theta= \pm\sigma e^4$  & \multirow{2}{*}{$\sigma\neq -1, 0$} &  \multirow{2}{*}{$n=\pm\frac{2}{\sigma}$} \\
	            &&$\omega=\pm(e^{12}-(\sigma+1)e^{34})$ &&\\
	            \hline
				
				\bottomrule
			\end{tabular}
	}}
	\caption{$4$-dimensional solvable LCS Lie algebras satisfying condition \eqref{unimodular-condition}.}
	\label{extensibles}
\end{table}

We explain in details how to extend one example. 

\begin{ejemplo}\label{Ejemplo r'2}
Let $\mathfrak r'_2$ be the Lie algebra with  structure constants $(0,0,-13+24,-14-23)$. 
According to \cite{ABP} this Lie algebra admits 4 non equivalent LCS structures up to Lie algebra automorphism 
$$ \left\{\begin{array}{l}
\theta = \sigma e^1+\tau e^2 \\
\omega =  e^{13} - \tau e^{14} -\frac{1+\tau^2}{1+\sigma} e^{24} \\
\text{with } \sigma\neq -1,0, \quad \tau>0
\end{array}\right. \quad \quad \;\;
\left\{\begin{array}{l}
\theta = -2 e^1 \\
\omega = \sigma e^{12} + e^{34} \\
\text{with } \sigma\neq 0
\end{array}\right.  $$
$$ \quad \left\{\begin{array}{l}
\theta = \tau e^2 \\
\omega = e^{13} - \tau e^{14} -(1+\tau^2) e^{24} \\
\text{with } \tau>0
\end{array}\right. \quad 
\left\{\begin{array}{l}
\theta = \sigma e^1 \\
\omega =  e^{13} -\frac{1}{1+\sigma} e^{24} \\
\text{with } \sigma\neq -1,0 
\end{array}\right.$$

\smallskip

We verify next if each LCS structure satisfies condition \eqref{unimodular-condition}.  If we consider the cases $\theta =\sigma e^1+\tau e^2 $, $\theta = \tau e^2$ or $\theta = -2 e^1$ the condition \eqref{unimodular-condition} does not hold, then the possible LCS extension will not be unimodular.
Finally we consider the LCS structure with Lee form $\theta = \sigma e^1$ and $\sigma\neq -1,0$. In this case \eqref{unimodular-condition} is satisfied for $2=n\sigma$.
Therefore for any $n\in \N$ and $\sigma=\frac 2n$, the LCS structure on $\mathfrak r'_2$ given by 
$$\left\{\begin{array}{l}
\theta = \sigma e^1 \\
\omega =  e^{13} -\frac{1}{1+\sigma} e^{24} \\
\text{with } \sigma\neq -1,0
\end{array}\right.$$
can be extended to the $(2n+4)$-dimensional unimodular Lie algebra
\[\g=\mathfrak r'_2\ltimes_\pi \R^{2n},\]
where $\pi$ is a suitable LCS representation.
It follows from \Cref{pi=tita-ro} that $\g$ admits a LCS structure $(\tilde\omega,\tilde\theta)$ given by $\tilde\omega|_\h=\omega$, $\tilde\omega|_{\R^{2n}}=\omega_0$, $\tilde\theta|_\h=\theta$ and $\tilde\theta|_{\R^{2n}}=0$, where $\omega_0$ is any symplectic form on $\R^{2n}$. To be more specific we consider $n=2$, then $\sigma=1$. Let $\pi:\mathfrak r'_2\to\mathfrak{gl}(4,\R)$ be the representation given by 
\[\pi(e_1)=\begin{pmatrix}
0 &&& \\
&-1&&\\
&&-1&\\
&&&0
\end{pmatrix}=-\frac12\I+\begin{pmatrix}
\frac12 &&& \\
&-\frac12&&\\
&&-\frac12&\\
&&&\frac12
\end{pmatrix},\]
and $\pi(e_2)=\pi(e_3)=\pi(e_4)=0$, in a basis $\{e_5,e_6,e_7,e_8\}$ of $\R^4$ with $\omega_0=e^{56}+e^{78}$. It is easy to see that $\pi$ is a LCS representation.  Then the $8$-dimensional Lie algebra $\g=\mathfrak r'_2\ltimes_\pi \R^4$ has the following Lie brackets
\[[e_1,e_3]=e_3, \quad [e_2,e_3]=e_4, \quad [e_1,e_6]=-e_6,\] 
\[[e_1,e_4]=e_4, \quad [e_2,e_4]=-e_3, \quad [e_1,e_7]=-e_7,\]
and the LCS structure is given by
$$\left\{\begin{array}{l}
\tilde\omega=e^{13}-\frac{1}{2}e^{24}+e^{56}+e^{78}\\
\tilde\theta= e^1.
\end{array}\right.$$
\end{ejemplo}

\begin{obs}
	Note that this example is not covered by the construction given in \cite[Proposition 1.8]{ABP}.  
\end{obs}

\

\section{Examples of solvmanifolds with LCS structures of the second kind}

\
In the section we use Theorem \ref{pi=tita-ro} and Theorem \ref{2tipo_cociente} to construct a family of solvmanifolds $\Gamma_m\backslash G$ admitting a LCS structure of the second kind, where $G$ is the simply connected Lie group associated to the Lie algebra considered in the Example \ref{Ejemplo r'2}. This Lie algebra can be decomposed as $\g=\R e_2\ltimes\R e_1\ltimes \R^{6}$ where the adjoint actions of $e_1$ and $e_2$ are given by
$$\ad_{e_1}=\begin{pmatrix}
1 &&&&& \\
&-1&&&&\\
&& 0&&&\\
&&&1 && \\
&&&&-1&\\
&&&&&0
\end{pmatrix}, \quad
\ad_{e_2}=\begin{pmatrix}
&&&& 0&&&\\
&&&&-1&&&\\
&&&& 0&&&\\
&&&& 0&&& \\
0&1&0& 0&0&0&0&\\
&&&& 0&&& \\
&&&& 0&&& 
\end{pmatrix}$$
in the reordered bases $\{e_3,e_6,e_5,e_4,e_7,e_8\}$ and $\{e_1,e_3,e_6,e_5,e_4,e_7,e_8\}$ respectively. The simply connected Lie group associated to $\g$ is 
$$G=\R e_2\ltimes_\psi(\R e_1\ltimes_\varphi \R^{6}),$$
where $\varphi(t)=\exp(t\ad_{e_1})$ and $\psi(t)=\exp(t\ad_{e_2})$. Next we build a lattice in $G$, to do this we start considering the Lie subgroup $H:=\R e_1\ltimes_\varphi \R^{6}$. Note that $H$ is an almost abelian Lie group. According to \cite{B} the Lie group $H$ admits a lattice if and only if there exists $t_0\in\R$, $t_0\neq0$, such that $\varphi(t_0)$ is conjugated to an integer matrix. 

We can write the matrix $\varphi(t)$ in the basis $\{e_3,e_6,e_5,e_4,e_7,e_8\}$ as 
$$\varphi(t)=\begin{pmatrix}
e^t &&&&& \\
&e^{-t} &&&&\\
&& 1&&&\\
&&&e^t  && \\
&&&&e^{-t}  &\\
&&&&&1
\end{pmatrix}.$$
We consider only the block
$$M=\begin{pmatrix}
e^t && \\
&e^{-t} &\\
&& 1
\end{pmatrix}.$$
The characteristic polynomial of the matrix $M$ is 
$$p(x)=(x-1)(x-e^t)(x-e^{-t}).$$
Fixing $m\in\mathbb{N}$, $m>2$, we define $t_m=\operatorname {arccosh}(\frac{m}{2})$, $t_m>0$, and we have that $(x-e^{t_m})(x-e^{-t_m})=x^2-mx+1\in\Z[x]$. Then the characteristic polynomial of $M$ for $t=t_m$ can be written as $p(x)=x^3-(m+1)x^2+(m+1)x-1$. Therefore, it is easy to see that $M$ is conjugated to the companion matrix $C_m$ of the polynomial $p$, that is, $M=Q_mC_mQ_m^{-1}$ where 
$$C_m=\begin{pmatrix}
0&0&1\\
1&0&-(1+m) \\
0&1&1+m 
\end{pmatrix} \quad \text{and} \quad 
Q_m=\begin{pmatrix}
1&e^{t_m}&e^{2t_m}\\
1&e^{-t_m}&e^{-2t_m} \\
1&1&1 
\end{pmatrix}.$$
We can copy this process in the second block of $\varphi(t)$ and we easily can check that 
$$\varphi(t_m)=P_mD_mP_m^{-1},$$ where
$$D_m=\begin{pmatrix}
C_m&0\\
0&C_m 
\end{pmatrix}\quad \text{and}\quad 
P_m=\begin{pmatrix}
Q_m&0\\
0&Q_m
\end{pmatrix}.$$
Therefore $\varphi(t_m)$ is conjugate to the integer matrix $D_m$. It follows from \cite{B} that $H$ admits a lattice  $$\Gamma_m=t_m\Z\ltimes P_m\Z^6.$$

\smallskip

We now write the matrix $\psi(t)$ in the basis $\{e_1,e_3,e_6,e_5,e_4,e_7,e_8\}$ as 
$$\psi(t)=\begin{pmatrix}
1 &0&0&0&0&0&0 \\
0&\cos(t) &0&0&-\sin(t)&0&0 \\
0&0&1&0&0&0&0\\
0&0&0&1&0&0&0\\
0&\sin(t) &0&0&\cos(t)&0&0 \\
0&0&0&0&0&1&0\\
0&0&0&0&0&0&1\\
\end{pmatrix}.$$

Clearly,  $\psi(2\pi)$ preserves the lattice $\Gamma_m$. Thus
$$\Lambda_m=2\pi\Z\ltimes\Gamma_m=2\pi\Z\ltimes (t_m\Z\ltimes P_m\Z^6)$$
is a lattice in $G$ for any $m>2$.
Therefore we obtain an explicit construction of examples of solvmanifolds $\Lambda_m\backslash G$ admitting a LCS structure of the second kind.

Since $\psi(2\pi)=\I$, we have that $\Lambda_m\simeq 2\pi\Z\times\Gamma_m$, and therefore it can be considered as a lattice in $G'=\R\times H$. Then, according to \cite[Theorem 3.6]{R} the corresponding solvmanifolds $M_m:=\Lambda_m\backslash G$ and $M_m':=\Lambda_m\backslash G'$ are diffeomorphic. Note that $M'$ is diffeomorphic to the product of $S^1$ and the solvmanifold $\Gamma_m\backslash H$. We also note that $G'$ can be seen as an almost abelian Lie group, more explicity, $G'=\R \ltimes_\rho \R^7$, where 
$$\rho(t)=\begin{pmatrix}
1 &&&&& \\
&e^{t} &&&&& \\
&&e^{-t} &&&&\\
&&& 1&&&\\
&&&&e^{t}  && \\
&&&&&e^{-t}  &\\
&&&&&&1
\end{pmatrix}.$$
It is easy to see that $\rho(t_m)$ is conjugated to the integer matrix 
$$R_m=\begin{pmatrix}
	1&0\\
	0&P_m
\end{pmatrix}.$$
Using this identification between $M_m$ and $M'_m$ we can prove that:

\begin{prop}
The solvmanifolds $\Lambda_m\backslash G$ are pairwise non homeomorphic.
\end{prop}

\begin{proof}
We assume that $M_m$ and $M_n$ are homeomorphic, then $M'_m$ and $M'_n$ are homeomorphic as well, and therefore their fundamental groups $\pi_1(M'_m)$ and $\pi_1(M'_n)$ are isomorphic.  Since $G'$ is simply connected, we have that these fundamental groups are isomorphic to the lattices, and therefore $\Lambda_m\cong\Lambda_n$. 
Since $G'$ is completely solvable, we can use the Saito’s
rigidity theorem \cite{Sai} to extend this isomorphism to an automorphism of $G'$.  
Since the lattices differ by an automorphism of $G'$, it follows from \cite[Theorem 3.6]{Hu} that the integer matrix $R_m$ is conjugated either to $R_n$ or to $R^{-1}_n$.
Finally comparing the eigenvalues of $R_m$ and $R_n$ we obtain that this happens if and only if $m=n$.
\end{proof}

We study now the de Rham and the Morse-Novikov cohomology of the solvmanifolds $\Lambda_m\backslash G\cong\Lambda_m\backslash G'$. We note first that the Lie algebra $\g'$ of $G'$ is given by $\g'=\R e_1\ltimes \R^{7}$ with 
$$\ad_{e_1}=\begin{pmatrix}
0 & &&&&&\\
&1 &&&&& \\
&&-1&&&&\\
&&& 0&&&\\
&&&&1 && \\
&&&&&-1&\\
&&&&&&0
\end{pmatrix}$$
in the basis $\{e_2,e_3,e_6,e_5,e_4,e_7,e_8\}$ of $\R^{7}$. Therefore the Lie group $G'$, unlike $G$, is completely solvable, that is, $\ad_X$ has real eigenvalues for all $X \in \g'$. 
Since $G'$ is completely solvable these cohomologies can be computed in terms of the Lie algebra cohomology. Indeed, it follows from \cite{Hat} (see also \cite{AO1}) that  $H^*_{dR}(\Lambda_m\backslash G')\cong H^*(\g') $ and  $H^*_\theta(\Lambda_m\backslash G')\cong H^*_\theta(\g')$ for any $m\in\N$.

According to \cite{Sa}, the $k^{th}$ Betti number of $\g'$, $\beta_k=\dim H^k(\g')$, can be computed in terms of the dimension of
$Z^j(\g')=\{\alpha\in\alt^j \g^* : d\alpha=0\}$ as follows:
\begin{equation}\label{betti} 
\beta_k=\dim H^k(\g')=\dim Z^k(\g')+\dim Z^{k-1}(\g')-\binom{8}{k-1},
\end{equation}
for $k>2$. Note that $\beta_0=1$ and $\beta_1=\dim(\g'/[\g',\g'])=4$. Using equation \eqref{betti} it can be seen that $\beta_2=10$, $\beta_3=20$ and $\beta_4=26$. Finally, due to Poincar\'e duality, we obtain that $\beta_5=20$, $\beta_6=10$, $\beta_7=4$ and $\beta_8=1$.

For the Morse-Novikov cohomology, the corresponding Betti numbers $\beta^\theta_k=\dim H_\theta^k(\g')$ satisfy a similar equation. Indeed, setting 
$Z^j_{\theta}(\g)=\{\alpha\in\alt^j \g^* : d_{\theta}\alpha=0\}$ we have
\begin{equation}\label{bettitilde}
\beta^ \theta_k=\dim H^k_{\theta}(\g')=\dim Z^k_{\theta}(\g')+\dim 
Z^{k-1}_{\theta}(\g')-\binom{8}{k-1},
\end{equation}
for $k>2$. It is easy to see that $\beta^ \theta_0=0$ and $\beta^ \theta_1=2$. Then, after some computations and using \eqref{bettitilde}, it can be seen that $\beta^ \theta_2=8$, $\beta^ \theta_3=14$, $\beta^ \theta_4=16$, $\beta^ \theta_5=14$, $\beta^ \theta_6=8$, $\beta^ \theta_2=2$ and $\beta^ \theta_8=0$.

\

\begin{small}
	\noindent{\itshape Acknowledgement.}
	I would like to thank A. Andrada for a careful reading of the first draft of the paper and useful remarks to improve the final version. Also thanks to D. Angella for his useful comments during his stay at Universidad Nacional de C\'ordoba. Finally, I would like to thanks I. Dotti for her useful suggestions to improve the exposition of this paper.
\end{small}


\bibliographystyle{plain}

\end{document}